\documentclass[12pt]{amsart}

\usepackage{amssymb, amscd, txfonts}
\usepackage[all]{xy}
\usepackage{graphicx}

\numberwithin{equation}{section}

\newtheorem{thm}{Theorem}

\newtheorem{lem}[thm]{Lemma}
\newtheorem{cor}[thm]{Corollary}
\newtheorem{conj}[thm]{Conjecture}

\theoremstyle{definition}

\theoremstyle{remark}

\renewcommand{\hom}{\operatorname{Hom}}
\renewcommand{\ker}{\operatorname{Ker}}

\newcommand{\Z}{\mathbb{Z}}

\newcommand{\C}{\mathbb{C}}
\newcommand{\F}{\mathbb{F}}

\DeclareMathOperator{\eul}{Eul}
\DeclareMathOperator{\im}{Im}

\DeclareMathOperator{\rank}{rank}

\DeclareMathOperator{\spin}{Spin}

\DeclareMathOperator{\tr}{tr}

\begin{document}

\title[Twisted Alexander polynomials and incompressible surfaces given by ideal points]
{Twisted Alexander polynomials and incompressible surfaces given by ideal points}
\author[T.~Kitayama]{Takahiro KITAYAMA}
\address{Department of Mathematics, Tokyo Institute of Technology,
2-12-1 Ookayama, Meguro-ku, Tokyo 152-8551, Japan}
\email{kitayama@math.titech.ac.jp}

\begin{abstract}
We study incompressible surfaces constructed by Culler-Shalen theory
in the context of twisted Alexander polynomials.
For a $1$st cohomology class of a $3$-manifold
the coefficients of twisted Alexander polynomials induce
regular functions on the $SL_2(\C)$-character variety.
We prove that if an ideal point gives
a Thurston norm minimizing non-separating surface dual to the cohomology class,
then the regular function of the highest degree has
a finite value at the ideal point.
\end{abstract}

\maketitle

\section{Introduction}

Culler and Shalen~\cite{CS} applied Tits-Bass-Serre theory~\cite{Se1, Se2}
to the functional field of the $SL_2(\C)$-representation variety
of a $3$-manifold,
and presented a method to construct essential surfaces in the $3$-manifold
from an ideal point of the $SL_2(\C)$-character variety.
On the basis of their theory, for example,
Morgan and Shalen~\cite{MS1, MS2, MS3} gave
new understanding of Thurston's fundamental works~\cite{Th2, Th3}
in the context of representations of a $3$-manifold group,
and Culler, Gordon, Luecke and Shalen~\cite{CGLS} proved
the cyclic surgery theorem on Dehn filling of knots.
We refer to the exposition \cite{Sh} for literature on Culler-Shalen theory.

Twisted Alexander polynomials~\cite{Li, W} of a $3$-manifold,
which is essentially equal to certain Reidemeister torsion~\cite{KL, Kitan},
induce regular functions on its character varieties.
The works~\cite{FV1, FV3} by Friedl and Vidussi showing
that twisted Alexander polynomials detect fiberedness of $3$-manifolds
and the Thurston norms of ones which are not closed graph manifolds
were breakthroughs.
Of particular interest on the invariants in recent years has been
properties and applications of the regular functions
on the $SL_2(\C)$-character variety~\cite{DFJ, KKM, KM, Kitay, KT, Mo}.
We refer to the survey paper \cite{FV2} for details and related topics
on twisted Alexander polynomials.  

This work was intended as an attempt to bring together the above two areas.
In general, it is difficult to figure out the isotopy class of an essential surface
constructed by Culler-Shalen theory.
In this paper we describe a necessary condition
on the regular functions induced by twisted Alexander polynomials
for such a surface to be of a certain type.
It is known that the homology classes
of the boundary components of such a surface can be determined
by trace functions for simple closed curves in $\partial M$~\cite{CCGLS, CGLS}.  
The important point to note here is that our result concerns
the homology class of such a surface itself.
The result extends the main theorem of \cite{Kitay} on knot complements
to general $3$-manifolds.

Let $M$ be a connected compact orientable irreducible $3$-manifold
with empty or toroidal boundary and let $\psi \in H^1(M; \Z)$ be nontrivial.
We denote by $X^{irr}(M)$ the Zariski closure of the subset
of the $SL_2(\C)$-character variety of $M$
consisting of the characters of irreducible representations.
For an irreducible component $X_0$ in $X^{irr}(M)$
we define an invariant $\mathcal{T}_\psi^{X_0} \in \C[X_0][t, t^{-1}]$
induced by refined torsion invariants in the sense of Turaev~\cite{Tu1, Tu2},
which is regarded as certain normalizations of twisted Alexander polynomials.
In the case where $M$ is a knot complement
$\mathcal{T}_\psi^{X_0}$ coincides with the invariant
introduced in \cite[Theorem 1.5 and Theorem 7.2]{DFJ},
and is called the \textit{torsion polynomial function} of $X_0$. 
From the corresponding property
of twisted Alexander polynomials~\cite[Theorem 1.1]{FK1}
the invariant $\mathcal{T}_\psi^{X_0}$ satisfies that
$\deg \mathcal{T}_\psi^{X_0} \leq 2 ~ ||\psi||_T$,
where $||\psi||_T$ is the Thurston norm of $\psi$.
For a curve $C$ in $X_0$ we write $\mathcal{T}_\psi^C \in \C[C][t, t^{-1}]$
for the restriction of $\mathcal{T}_\psi^{X_0}$ to $C$,
and set $c(\mathcal{T}_\psi^C) \in \C[C]$ to be the coefficient function
in $\mathcal{T}_\psi^C$ of the highest degree $2 ~ ||\psi||_T$.

We suppose that $X^{irr}(M)$ has an irreducible component of positive dimension.
For an ideal point $\chi$ of a curve in $X^{irr}(M)$
we say that \textit{$\chi$ gives a surface $S$} in $M$
if $S$ is constructed from $\chi$ by Culler-Shalen theory \cite[Section 2]{CS}
as described in Section $2$. 
The main theorem of this paper is as follows:
\begin{thm} \label{thm_main}
Let $\psi \in H^1(M; \Z)$ be nontrivial.
Suppose that an ideal point $\chi$ of a curve $C$ in $X^{irr}(M)$ gives
a surface $S$ in $M$ satisfying the following:
\begin{enumerate}
\item[\upshape{(1)}] The homology class of $S$ is dual to $\psi$.
\item[\upshape{(2)}] $S$ is Thurston norm minimizing.
\item[\upshape{(3)}] The surface obtained by identifying components of $S$ parallel to each other is non-separating.
\end{enumerate}
Then  $c(\mathcal{T}_\psi^C)(\chi)$ is finite.
\end{thm}

An initial motivation of this work came from the following conjecture
by Dunfield, Friedl and Jackson~\cite[Conjecture 8.9]{DFJ}.
\begin{conj}[{\cite[Conjecture 8.9]{DFJ}}] \label{conj_DFJ}
Let $M$ be the exterior of a knot $K$ in a homology $3$-sphere
and let $\psi \in H^1(M; \Z)$ be a generator.
If an ideal point $\chi$ of a curve $C$ in $X^{irr}(M)$ gives
a Seifert surface of $K$, then the leading coefficient of $\mathcal{T}_\psi^C$
has a finite value at $\chi$.
\end{conj}
In \cite{Kitay} the author gave a partial affirmative answer to the conjecture.
Now the result is a direct corollary of Theorem \ref{thm_main}.
\begin{cor}[{\cite[Theorem 1.2]{Kitay}}]
Let $M$ be the exterior of a knot $K$ in a homology $3$-sphere
and let $\psi \in H^1(M; \Z)$ be a generator.
If an ideal point $\chi$ of a curve $C$ in $X^{irr}(M)$ gives
a minimal genus Seifert surface of $K$,
then $c(\mathcal{T}_\psi^C)(\chi)$ is finite.
\end{cor}

Note that an essential surface is not necessarily Thurston norm minimizing,
and that if $\deg \mathcal{T}_\psi^C < 2 ~ ||\psi||_T$,
then $c(\mathcal{T}_\psi^C) = 0$,
but the leading coefficient of $\mathcal{T}_\psi^C$ is not necessarily bounded.
Thus we can ask the following two questions:
\begin{enumerate}
\item Can condition $(2)$ in Theorem \ref{thm_main} be eliminated?
\item Can the conclusion `$c(\mathcal{T}_\psi^C)(\chi)$ is finite' of Theorem \ref{thm_main} be replaced by `the leading coefficient of $\mathcal{T}_\psi^C$ has a finite value at $\chi$'?
\end{enumerate}

The proof of Theorem \ref{thm_main} is based
on the following three key observations.
We denote by $S'$ the non-separating surface in condition $(3)$
and by $N$ the complement of an open tubular neighborhood of $S'$
identified with $S' \times (-1, 1)$.
Since the surface $S$ is essential, we can regard $\pi_1 N$
as a subgroup of $\pi_1 M$.
First, for an irreducible representation $\rho \colon \pi_1 M \to SL_2(\C)$,
the coefficient of the highest degree in a certain normalization
of the twisted Alexander polynomial associated to $\psi$ and $\rho$ coincides
with the torsion invariant of the pair $(N, S' \times 1)$
associated to $\rho|_{\pi_1 N}$.
This is proved by a surgery formula of torsion invariants.
Second, by the definition of torsion invariants the regular function
on the curve $C$ induced by the above torsion invariant is written
as a polynomial in $\{ I_\gamma \}_{\gamma \in \pi_1 N}$,
where $I_\gamma \colon C \to \C$ is the trace function
associated to $\gamma \in \pi_1 M$.
Thirdly, it follows from \cite[Theorem 2.2.1]{CS}
that $I_\gamma$ does not have a pole at the ideal point $\chi$
for all $\gamma \in \pi_1 N$.
Combining these observations, we can prove Theorem \ref{thm_main}.
See Section $4$ for the details of the proof.

This paper is organized as follows.
Section $2$ contains a brief overview of the construction of essential surfaces
from an ideal point by Culler and Shalen.
In Section $3$ we define the torsion polynomial functions
for general $3$-manifolds. 
Section $4$ is devoted to the proof of Theorem \ref{thm_main} along the above outline.
In this paper we mainly work with Reidemeister torsion
rather than twisted Alexander polynomials.

\subsection*{Acknowledgment}

The author would like to thank the anonymous referee
for helpful suggestions in revising the manuscript.
This research was supported by JSPS Research Fellowships for Young Scientists.

\section{Ideal points and essential surfaces}

We begin with setting up notation, and briefly review
the construction of essential surfaces from ideal points
of the $SL_2(\C)$-character variety by Culler and Shalen \cite{CS}.
For a thorough treatment we refer to the exposition \cite{Sh}.
For more details on character varieties we also refer to \cite{LM}.

Let $M$ be a connected compact orientable irreducible $3$-manifold.
We set $R(M) = \hom(\pi_1 M, SL_2(\C))$, which naturally has the structure
of an affine algebraic set over $\C$.
The algebraic group $SL_2(\C)$ acts on $R(M)$ by conjugation of representations,
and we denote by $X(M)$ the algebro-geometric quotient
$\hom(\pi_1 M, SL_2(\C)) // SL_2(\C)$
and by $t \colon R(M) \to X(M)$ the quotient map.
The affine algebraic set $X(M)$ is called
the \textit{$SL_2(\C)$-character variety} of $M$.
Its coordinate ring $\C[X(M)]$ is equal
to the subring $\C[R(M)]^{SL_2(\C)}$ of $\C[R(M)]$
consisting of regular functions which are invariant under conjugation.
For $\rho \in R(M)$ its \textit{character} $\chi_\rho \colon \pi_1 M \to \C$ is
defined by $\chi_\rho(\gamma) = \tr \rho(\gamma)$ for $\gamma \in \pi_1 M$.
It is known that $X(M)$ is realized as the set of the characters $\chi_\rho$
for $\rho \in R(M)$ so that $t(\rho) = \chi_\rho$.
For $\gamma \in \pi_1 M$ the \textit{trace function}
$I_\gamma \colon X(M) \to \C$ is defined
by $I_\gamma(\chi_\rho) = \tr \rho(\gamma)$ for $\rho \in R(M)$.
It is known that $\C[X(M)]$ is generated
by $\{ I_\gamma \}_{\gamma \in \pi_1 M}$.
We denote by $X^{irr}(M)$ the Zariski closure of the subset of $X(M)$
consisting of the characters of irreducible representations. 

Suppose $X(M)$ has an irreducible component of positive dimension,
and let $C$ be a curve in $X(M)$.
For the smooth projective model $\widetilde{C}$ of $C$
the points where the rational map $\widetilde{C} \to C$ is undefined are called
the \textit{ideal points} of $C$.
Let $\chi$ be an ideal point of $C$.
By \cite[Proposition 1.4.4]{CS} there exists a curve $D$ in $t^{-1}(C)$
such that $t|_D$ is not constant.
The map $t|_D$ extends
to a regular map $\widetilde{t|_D} \colon \widetilde{D} \to \widetilde{C}$
so that the inverse image of ideal points of $C$ consists of ones of $D$.
Thus we have the following commutative diagram of rational maps:
\[
\begin{CD}
\widetilde{D} @>>> D @>>> R(M) \\
@VV \widetilde{t|_D} V @VV t|_D V @VV t V\\
\widetilde{C} @>>> C @>>> X(M) 
\end{CD}
\]
Let $\tilde{\chi} \in \widetilde{D}$ be an ideal point of $D$
with $\widetilde{t|_D}(\tilde{\chi}) = \chi$.
From Bass-Serre theory~\cite{Se1, Se2} associated
to the discrete valuation of $\C(D)$ at $\tilde{\chi}$
there exists a canonical action of $SL_2(\C(D))$ on a tree $T_{\tilde{\chi}}$
without inversions.
Here an action of a group on a tree is said to be \textit{without inversions}
if the action does not reverse the orientation of any invariant edge.
We denote by $\tilde{\rho} \colon \pi_1 M \to SL_2(\C(D))$
the tautological representation, which is defined
by $\tilde{\rho}(\gamma)(\rho) = \rho(\gamma)$
for $\gamma \in \pi_1 M$ and $\rho \in D$.
Pulling back the action by $\tilde{\rho}$, we have an action
of $\pi_1 M$ on $T_{\tilde{\chi}}$ without inversions.
A heart of Culler-Shalen theory is that the action is nontrivial,
i.e., for any vertex of $T_{\tilde{\chi}}$
the stabilizer subgroup of $\pi_1 M$ is proper~\cite[Theorem 2.2.1]{CS}.
Essentially due to Stallings, Epstein and Waldhausen
it follows from the nontriviality that there exists
a map $f \colon M \to T_{\tilde{\chi}} / \pi_1 M$
such that $f^{-1}(P)$ is an essential surface,
where $P$ is the set of the midpoints of edges in $T_{\tilde{\chi}} / \pi_1 M$.
We recall that a non-empty properly embedded compact orientable surface $S$
in $M$ is called \textit{essential} if for any component $S_0$ of $S$
the inclusion-induced homomorphism $\pi_1 S_0 \to \pi_1 M$ is injective,
and $S_0$ is not homeomorphic to $S^2$ nor a boundary parallel surface.
We say that \textit{$\chi$ gives a surface $S$} if $S = f^{-1}(P)$
for some $f$ as above.

\section{Torsion polynomial functions}

We first review Reidemeister torsion and its refinement by an Euler structure
introduced by Turaev~\cite{Tu1, Tu2}.
Then we define the torsion polynomial functions for general $3$-manifolds,
following the constructions of these for knot complements
by Dunfield, Friedl and Jackson~\cite[Theorem 1.5 and Theorem 7.2]{DFJ}.
For basics of topological torsion invariants we also refer the reader
to the expositions \cite{Mi, N}.

Let $C_* = (C_n \xrightarrow{\partial_n} C_{n-1} \to \cdots \to C_0)$ be
a finite dimensional chain complex over a field $\F$,
and let $c = (c_i)_i$ and $h = (h_i)_i$ be bases of $C_i$ and $H_i(C_*)$
respectively.
We choose a basis $b_i$ of $\im \partial_{i+1}$ for each $i$.
Taking a lift of $h_i$ in $\ker \partial_i$ and combining it with $b_i$,
we have a basis $b_i h_i$ of $\ker \partial_i$ for each $i$.
Then taking a lift of $b_{i-1}$ in $C_i$ and combining it with $b_i h_i$,
we have a basis $b_i h_i b_{i-1}$ of $C_i$ for each $i$.
The \textit{algebraic torsion} $\tau(C_*, c, h)$ is defined as:
\[ \tau(C_*, c, h) := \prod_{i=0}^n [b_i h_i b_{i-1} / c_i]^{(-1)^{i+1}} ~\in \F^\times, \]
where $[b_i h_i b_{i-1} / c_i]$ is the determinant of the base change matrix
from $c_i$ to $b_i h_i b_{i-1}$ for each $i$.
It can be checked that $\tau(C_*, c, h)$ does not depend
on the choices of $(b_i)_i$ and $(b_i h_i b_{i-1})_i$.
When $C_*$ is acyclic,
we just write $\tau(C_*, c)$ for $\tau(C_*, c, \emptyset)$.

Let $Y$ be a connected finite CW-complex
and let $Z$ be a proper subcomplex of $Y$.
We denote by $\widetilde{Y}$ the universal cover of $Y$
and by $\widetilde{Z}$ the pullback of $Z$
by the covering map $\widetilde{Y} \to Y$.
For a representation $\rho \colon \pi_1 Y \to GL_n(\F)$
we define the twisted homology group and the twisted cohomology group as:
\begin{align*}
H_i^\rho(Y, Z; \F) &:= H_i(C_*(\widetilde{Y}, \widetilde{Z}) \otimes_{\Z[\pi_1 Y]} \F^n), \\
H_\rho^i(Y, Z; \F) &:= H^i(\hom_{\Z[\pi_1 Y]}(C_*(\widetilde{Y}, \widetilde{Z}), \F^n)).
\end{align*}
When $Z$ is empty, we just write $H_i^\rho(Y; \F^n)$ and $H_\rho^i(Y; \F^n)$
for $H_i^\rho(Y, \emptyset; \F^n)$ and $H_\rho^i(Y, \emptyset; \F^n)$
respectively.

For a representation $\rho \colon \pi_1 Y \to GL_n(\F)$
and a basis $h$ of $H_*^\rho(Y, Z; \F^n)$
the \textit{Reidemeister torsion} $\tau_\rho(Y, Z; h)$
associated to $\rho$ and $h$ is defined as follows.
We choose a lift $\{ \tilde{e}_i \}$ of cells
of $Y \setminus Z$ in $\widetilde{Y}$.
Then
\[ \tau_\rho(Y, Z; h) := \tau(C_*(\widetilde{Y}, \widetilde{Z}) \otimes_{\Z[\pi_1 Y]} \F^n, \langle \tilde{e}_i \otimes f_j \rangle_{i, j}, h) ~\in \F^\times / (-1)^n \det \rho(\pi_1 Y), \]
where $\langle f_1, \dots, f_n \rangle$ is the standard basis of $\F^n$.
It is checked that $\tau_\rho(Y, Z; h)$ is invariant
under conjugation of representations,
and is known that $\tau_\rho(Y, Z; h)$ is a simple homotopy invariant.
When $Z$ is empty or when $H_*^\rho(Y, Z; \F^n) = 0$,
then we drop $Z$ or $h$ in the notation $\tau_\rho(Y, Z; h)$.

Now we suppose that $\chi(Y) = 0$.
For two lifts $\{ \tilde{e}_i \}$ and $\{ \tilde{e}_i' \}$
of cells of $Y$ in $\widetilde{Y}$ we set
\[ \{ \tilde{e}_i' \} / \{ \tilde{e}_i \} := \sum_i (-1)^{\dim e_i} \tilde{e}_i' / \tilde{e}_i ~ \in H_1(Y; \Z), \]
where $\tilde{e}_i' / \tilde{e}_i$ is the element $[\gamma] \in H_1(Y; \Z)$
for $\gamma \in \pi_1 Y$ such that $\tilde{e}_i' = \gamma \cdot \tilde{e}_i$.
Two lifts $\{ \tilde{e}_i \}$ and $\{ \tilde{e}_i' \}$ are called equivalent
if $\{ \tilde{e}_i' \} / \{ \tilde{e}_i \} = 0$.
An equivalence class of a lift $\{ \tilde{e}_i\}$ is
called an \textit{Euler structure} of $Y$ and we denote by $\eul(Y)$
the set of Euler structures.
For $h \in H_1(Y; \Z)$ and $[\{ \tilde{e}_i \}] \in \eul(Y)$ we set
$[h \cdot \{ \tilde{e}_i \}]$ to be the class of some lift $\{ \tilde{e}_i \}$
with $\{ \tilde{e}_i' \} / \{ \tilde{e}_i \} = h$.
This defines a free and transitive action of $H_1(Y; \Z)$ on $\eul(Y)$.
For a subdivision $Y'$ it is checked
that the natural map $\eul(Y') \to \eul(Y)$ is a $H^1(M)$-equivalent bijection.

For a representation $\rho \colon \pi_1 Y \to GL_n(\F)$
and $\mathfrak{e} \in \eul(Y)$ we define
the \textit{refined Reidemeister torsion} $\tau_\rho(Y; \mathfrak{e})$
associated to $\rho$ and $\mathfrak{e}$ is defined as follows.
We choose a lift $\{ \tilde{e}_i \}$ of cells of $Y$ in $\widetilde{Y}$
representing $\mathfrak{e}$.
If $H_*^\rho(Y; \F^n) = 0$, then we define
\[ \tau_\rho(Y; \mathfrak{e}) := \tau(C_*(\widetilde{Y}) \otimes_{\Z[\pi_1 Y]} \F^n, \langle \tilde{e}_i \otimes f_j \rangle_{i, j}, h) ~\in \F^\times / \langle (-1)^n \rangle, \]
and if $H_*^\rho(Y; \F^n) \neq 0$, then we set $\tau_\rho(Y; \mathfrak{e}) = 0$.
It is checked that $\tau_\rho(Y; \mathfrak{e})$ is also invariant
under conjugation of representations, and is known
that $\tau_\rho(Y; \mathfrak{e})$ is invariant
under subdivisions of CW-complexes.
It is straightforward from the definitions to see
that if $H_*^\rho(Y; \F^n) = 0$, then
\[ \tau_\rho(Y) = [\tau_\rho(Y; \mathfrak{e})] ~\in \F^\times / (-1)^n \det \rho(\pi_1 Y). \]
Note that the ambiguity $\langle (-1)^n \rangle$ can be also eliminated
if a \textit{homology orientation} is fixed.
Since the case of $n=2$ is our focus here,
we do not touch that topic in this paper.

In the following we consider torsion invariants
of a connected compact orientable $3$-manifold $M$
with empty or toroidal boundary.
Let $\psi \in H^1(M; \Z)$ be nontrivial.
By abuse of notation, we use the same letter $\psi$
for the homomorphism $\pi_1 M \to \langle t \rangle$ corresponding to $\psi$,
where $\langle t \rangle$ is the infinite cyclic group
generated by the indeterminate $t$.
For a representation $\rho \colon \pi_1 M \to GL_n(\F)$
a representation $\psi \otimes \rho \colon \pi_1 M \to GL_n(\F(t))$ is defined
by $(\psi \otimes \rho)(\gamma) = \psi(\gamma) \rho(\gamma)$
for $\gamma \in \pi_1 M$.
If $H_*^{\psi \otimes \rho}(M; \F(t)^n) = 0$,
then the Reidemeister torsion
$\tau_{\psi \otimes \rho}(M) \in \F(t)^\times / (-1)^n \det (\psi \otimes \rho)(\pi_1 M)$
is defined, and was shown
by Kirk and Livingston~\cite{KL}, and Kitano~\cite{Kitan}
to be essentially equal to the \textit{twisted Alexander polynomial}
associated to $\psi$ and $\rho$.
We refer the reader to the survey paper \cite{FV2}
for details and related topics on twisted Alexander polynomials.
Friedl and Kim~\cite[Theorem 1.1 and Theorem 1.2]{FK1} showed that
\[ \deg \tau_{\psi \otimes \rho}(M) \leq n ~ ||\psi||_T \]
and that if $\psi$ is represented by a fiber bundle $M \to S^1$
and if $M \neq S^1 \times D^2, S^1 \times S^2$, then 
\[ \deg \tau_{\psi \otimes \rho}(M) = n ~ || \psi ||_T \]
and $\tau_{\psi \otimes \rho}(M)$ is represented
by a fraction of monic polynomials over $\F$
(cf. \cite{FK2}).
Here $|| \psi ||_T$ is the Thurston norm of $\psi$~\cite{Th1},
which is defined to be the minimum of $\sum_{i=1}^l \max \{ -\chi(S_i), 0 \}$
for a properly embedded surface dual to $\psi$
with the components $S_1, \dots, S_l$.
The second statement for fibered knot complements has been shown
by Cha~\cite{C}, and Goda, Kitano and Morifuji \cite{GKM}.

The following lemma is an extension of \cite[Theorem 1.5 and Theorem 7.2]{DFJ}
on knot complements to general $3$-manifolds.
\begin{lem} \label{lem_T}
Let $X_0$ be an irreducible component of $X^{irr}(M)$.
There is an invariant $\mathcal{T}_\psi^{X_0} \in \C[X_0][t, t^{-1}]$,
which is the refined Reidemeister torsion
associated to a representation of $\pi_1 M$ and an Euler structure of $M$,
satisfying the following
for all representations $\rho \colon \pi_1 M \to SL_2(\C)$
with $\chi_\rho \in X_0$:
\begin{enumerate}
\item[\upshape{(1)}] If $H_*^{\psi \otimes \rho}(M; \C(t)^2) = 0$ then, $\mathcal{T}_\psi^{X_0}(\chi_\rho) = \tau_{\psi \otimes \rho}(M) \in \C(t) / \langle t \rangle$.
\item[\upshape{(2)}] If $H_*^{\psi \otimes \rho}(M; \C(t)^2) \neq 0$ then, $\mathcal{T}_\psi^{X_0}(\chi_\rho) = 0$.
\item[\upshape{(3)}] $\mathcal{T}_\psi^{X_0}(\chi_\rho)(t^{-1}) = \mathcal{T}_\psi^{X_0}(\chi_\rho)(t)$.
\end{enumerate}
\end{lem}
We call $\mathcal{T}_\psi^{X_0}$ in Lemma \ref{lem_T}
the \textit{torsion polynomial function} of $X_0$. 
For a curve $C$ in $X_0$ we denote by $\mathcal{T}_\psi^C \in \C[X_0][t, t^{-1}]$
the restriction of $\mathcal{T}_\psi^{X_0}$ to $C$,
and by $c(\mathcal{T}_\psi^C) \in \C[C]$ the coefficient function
in $\mathcal{T}_\psi^C$ of the highest degree $2 ~ ||\psi||_T$.

Before the proof we recall the relation
between $\spin^c$-structures and Euler structures for $3$-manifolds.
We denote by $\spin^c(M)$ the set of $\spin^c$-structures of $M$.
The set $\spin^c(M)$ admits a canonical free and transitive action
by $H_1(M; \Z)$.
Given $\mathfrak{s} \in \spin^c(M)$, we can consider
the Chern class $c_1(\mathfrak{s}) \in H^2(M, \partial M; \Z)$, and we have
\[ c_1(h \cdot \mathfrak{s}) = c_1(\mathfrak{s}) + 2h \]
for $h \in H_1(M; \Z)$ under the identification
$H^2(M, \partial M; \Z) = H_1(M; \Z)$ by the Poincar\'e duality.
Turaev showed that there exists a canonical $H_1(M; \Z)$-equivariant bijection
between $\spin^c(M)$ and $\eul(M)$ for a triangulation of $M$.
See \cite[Section XI.1]{Tu2} for full details.

\begin{proof}[Proof of Lemma \ref{lem_T}]
Let $X_0$ be an irreducible component of $X^{irr}(M)$.
By \cite[Proposition 1.4.4]{CS} there exists
an irreducible component $R_0$ of $R(M)$ such that $t(R_0) = X_0$.
We regard the tautological representation
$\tilde{\rho} \colon \pi_1 M \to SL_2(\C[R_0])$,
which is defined as in Section $2$,
as a representation $\pi_1 M \to SL_2(\C(R_0))$.
Since the subspace of $R_0$ consisting of irreducible representations is dense,
$\tilde{\rho}$ is also irreducible.  
We choose $\mathfrak{e} \in \eul(M)$ corresponding
to $\mathfrak{s} \in \spin^c(M)$.
Then we set
\begin{equation} \label{eq_T1}
\mathcal{T} = \psi(c_1(\mathfrak{s})) \cdot \tau_{\psi \otimes \tilde{\rho}}(M; \mathfrak{e}) ~\in \C(R_0)(t),
\end{equation}
where $\psi(c_1(\mathfrak{s})) \in \langle t \rangle$ is
the image of $c_1(\mathfrak{s})$ by the homomorphism
$H_1(M) \to \langle t \rangle$ induced
by $\psi \colon \pi_1 M \to \langle t \rangle$.
Since $\tilde{\rho}$ is irreducible, by \cite[Theorem A.1]{FKK}
we see that $\mathcal{T} \in \C(R_0)[t, t^{-1}]$.
It follows from \eqref{eq_T1} that
\begin{equation} \label{eq_T2}
\mathcal{T}(\rho) = \psi(c_1(\mathfrak{s})) \cdot \tau_{\psi \otimes \rho}(M; \mathfrak{e}) \in ~\C[t, t^{-1}]
\end{equation}
for all $\rho \in R_0$.
In particular, the coefficients of $\mathcal{T}(\rho)$ have well-defined values
for all $\rho \in R_0$, and hence $\mathcal{T} \in \C[R_0][t, t^{-1}]$.
Furthermore, Reidemeister torsion is invariant
under conjugation of representations, and so we see
that $\mathcal{T} \in \C[R_0]^{SL_2(\C)}[t, t^{-1}]$, which implies
that $\mathcal{T}$ descends to an element of $\C[X_0][t, t^{-1}]$.
We define $\mathcal{T}_\psi^{X_0}$ to be this element.
Conditions $(1)$ and $(2)$ can be checked from \eqref{eq_T2}.
It follows from \cite[Theorem 1.5]{FKK} and the proof
that for irreducible $\rho \in R_0$ we have
\begin{equation} \label{eq_T3}
\mathcal{T}_\psi^{X_0}(\chi_\rho)(t^{-1}) = \mathcal{T}_\psi^{X_0}(\chi_\rho)(t).
\end{equation}
Since the subset of $\chi_\rho$ for irreducible $\rho \in R_0$ is
dense in $X_0$, \eqref{eq_T3} holds for all $\chi_\rho \in X_0$,
which shows condition $(3)$.
\end{proof}

\section{Proof of the main theorem}

Now we prove Theorem \ref{thm_main}.
Let $M$ be a connected compact orientable irreducible $3$-manifold
with empty or toroidal boundary and let $\psi \in H^1(M; \Z)$ be nontrivial.
Suppose that $X^{irr}(M)$ contains a curve $C$
and suppose that an ideal point $\chi$ of $C$ gives
an essential surface $S$ in $M$ satisfying the following:
\begin{enumerate}
\item[\upshape{(1)}] The homology class of $S$ is dual to $\psi$.
\item[\upshape{(2)}] $S$ is Thurston norm minimizing.
\item[\upshape{(3)}] The surface obtained by identifying components of $S$ parallel to each other is non-separating.
\end{enumerate}
We need to show that $c(\mathcal{T}_\psi^C)(\chi)$ is finite.

If $\deg \mathcal{T}_\psi^C < 2 ~ ||\psi||_T$,
then $c(\mathcal{T}_\psi^C)(\chi) = 0$.
Therefore in the following we can also suppose
that $\deg \mathcal{T}_\psi^C = 2 ~ ||\psi||_T$,
which holds if and only if $H_*^{\psi \otimes \rho}(M; \C(t)^2) = 0$
and $\deg \tau_{\psi \otimes \rho}(M) = 2 ~ ||\psi||_T$
for all but finitely many irreducible representations
$\rho \colon \pi_1 M \to SL_2(\C)$ with $\chi_\rho \in C$.

We denote by $S'$ the surface in condition $(3)$
with its components labeled $S_1, \dots, S_l$.
Note that since $S$ is essential, so is $S'$.
We identify a tubular neighborhood of $S'$ in $M$ with $S' \times [-1, 1]$,
and set $N := M \setminus S' \times (-1, 1)$.
We denote by $\iota_\pm \colon S' \to N$ the natural embeddings
such that $\iota_\pm(S') = S' \times (\pm 1)$.
Since the inclusion induced homomorphisms $\pi_1 N \to \pi_1 M$
and $\pi_1 S_i \to \pi_1 M$ for all $i$ are all injective,
in the following we identify $\pi_1 N$ and $\pi_1 S_i$ with their images.
(More precisely, for such identifications we need to fix paths
connecting base points of subspaces to the one in $M$.
Also in considering the twisted homology and cohomology groups of subspaces
such paths are understood to be chosen.
See, for instance, \cite[Section 2.1]{FK1} for details on a general treatment.)

Taking appropriate triangulations of $M$, $N$ and $S'$
and lifts of simplices in the universal covers,
we have the following exact sequences of twisted chain complexes
for a representation $\rho \colon \pi_1 M \to SL_2(\C)$:
\begin{gather}
\label{eq_exact1}
0 \to \bigoplus_{i=1}^l C_*(\widetilde{S_i}) \otimes \C(t)^2 \xrightarrow{t (\iota_+)_* - (\iota_-)_*} C_*(\widetilde{N}) \otimes \C(t)^2 \to C_*(\widetilde{M}) \otimes \C(t)^2 \to 0, \\
\label{eq_exact2}
0 \to \bigoplus_{i=1}^l C_*(\widetilde{S_i}) \otimes \C^2 \xrightarrow{(\iota_+)_*} C_*(\widetilde{N}) \otimes \C^2 \to C_*(\widetilde{N}, \widetilde{S' \times 1}) \otimes \C^2 \to 0,
\end{gather}
where the local coefficients in the first and second exact sequences are
understood to be induced by $\psi \otimes \rho$ and $\rho$ respectively.

First, we prove that
for an irreducible representation $\rho \colon \pi_1 M \to SL_2(\C)$
such that $H_*^{\psi \otimes \rho}(M; \C(t)^2) = 0$
and $\deg \tau_{\psi \otimes \rho}(M) = 2 ~ ||\psi||_T$ the homomorphism
$(\iota_+)_* \colon \bigoplus_{i=1}^l H_*^{\rho}(S_i; \C^2) \to H_*^\rho(N; \C^2)$
is an isomorphism and $H_*^\rho(N, S \times 1; \C^2) = 0$.
The second assertion follows from the first one
and the homology long exact sequence of \eqref{eq_exact2}.
Since $H_*^{\psi \otimes \rho}(M; \C(t)^2) = 0$,
it follows from the homology long exact sequence of \eqref{eq_exact1} that
$(t (\iota_+)_* - (\iota_-)) \colon \bigoplus_{i=1}^l H_*^\rho(S_i; \C^2) \otimes \C(t) \to H_*^\rho(N; \C^2) \otimes \C(t)$
is an isomorphism.
Hence
\begin{equation} \label{eq_rank}
\rank \bigoplus_{i=1}^l H_*^\rho(S_i; \C^2) = \rank H_*^\rho(N; \C^2).
\end{equation}
Since $(N, S \times 1)$ is homotopy equivalent to a CW pair
with all vertices in $S \times 1$ and without $3$-cells,
$H_0^\rho(N, S \times 1) = H_3^\rho(N, S \times 1) = 0$.
Hence it follows from the homology long exact sequence of \eqref{eq_exact2} that
$(\iota_+)_* \colon \bigoplus_{i=1}^l H_j^\rho(S_i; \C^2) \to H_j^\rho(N; \C^2)$
is surjective for $j =0$ and is injective for $j=2$.
From \eqref{eq_rank} the homomorphisms are isomorphisms for $j = 0, 2$.
The assertion for $j = 1$ is proved by techniques developed in \cite{FK1}
in terms of twisted Alexander polynomials.
Since $H_*^{\psi \otimes \rho}(M; \C(t)^2) = 0$
and $\deg \tau_{\psi \otimes \rho}(M) = 2 ~ ||\psi||_T$,
it follows from the proof of \cite[Theorem 1.1]{FK1}
that the inequalities in \cite[Proposition 3.3]{FK1} turn into equalities.
Now it follows from the proof of \cite[Proposition 3.3]{FK1} that
$(\iota_+)_* \colon \bigoplus_{i=1}^l H_1^\rho(S_i; \C^2) \to H_1^\rho(N; \C^2)$
is an isomorphism.

Second, we prove that for an irreducible representation
$\rho \colon \pi_1 M \to SL_2(\C)$ such that
$H_*^{\psi \otimes \rho}(M; \C(t)^2) = 0$
and $\deg \tau_{\psi \otimes \rho}(M) = 2 ~ ||\psi||_T$,
the following formula holds:
\begin{equation} \label{eq_prod}
\tau_{\psi \otimes \rho}(M) = \tau_\rho(N, S' \times 1) \prod_{j=0}^2 \det(t \cdot id - \iota_j)^{(j+1)}, 
\end{equation}
where $\iota_j$ denotes the isomorphism
$(\iota_+)_*^{-1} \circ (\iota_-)_* \colon \bigoplus_{i = 1}^l H_j^\rho(S_i; \C^2) \to H_j^\rho(N; \C^2)$
for each $j$.
We pick a basis $h$ of $H_*^\rho(S; \C^2)$.
By the multiplicativity of Reidemeister torsion~\cite[Theorem 3.1]{Mi} we have
\begin{align}
\label{eq_prod1}
\tau_{\psi \otimes \rho}(N; (\iota_+)_*(h \otimes 1)) \prod_{j=0}^2 \det(t \cdot id - (\iota_j)_*) &= \tau_{\psi \otimes \rho}(S; h \otimes 1) \tau_{\psi \otimes \rho}(M), \\
\label{eq_prod2}
\tau_\rho(N; (\iota_+)_*(h \otimes 1)) &= \tau_\rho(S; h) \tau_{\rho_0}(N, S \times 1).
\end{align}
By the functoriality of Reidemeister torsion~\cite[Proposition 3.6]{Tu1} we have
\begin{align}
\label{eq_prod3}
\tau_{\alpha \otimes \rho}(N; (\iota_+)_*(h \otimes 1)) &= \tau_{\rho}(N; (\iota_+)_*(h)), \\
\label{eq_prod4}
\tau_{\alpha \otimes \rho}(S; h \otimes 1) &= \tau_{\rho}(S; h).
\end{align}
The formula \eqref{eq_prod} follows
from \eqref{eq_prod1}, \eqref{eq_prod2}, \eqref{eq_prod3}, \eqref{eq_prod4}.

Thirdly, we prove that there exists a regular function $\varphi$ of $C$
in the subring of $\C[C]$ generated
by $\{ I_\gamma \}_{\gamma \in \pi_1 N}$ satisfying that
\begin{equation} \label{eq_F}
\varphi(\chi_\rho) = \tau_\rho(N, S \times 1)
\end{equation} 
for all irreducible representations $\rho \colon \pi_1 M \to SL_2(\C)$
with $\chi_\rho \in C$ such that $H_*^{\psi \otimes \rho}(M; \C(t)^2) = 0$
and $\deg \tau_{\psi \otimes \rho}(M) = 2 ~ ||\psi||_T$.
Let $\rho_0$ be such a representation.
We take a finite $2$-dimensional CW-pair $(C, W)$ with $C_0(V, W) = 0$
which is simple homotopy equivalent to $(N, S \times 1)$,
and we identify $\pi_1 V$ with $\pi_1 N$ by the homotopy equivalence.
The differential map
$C_2(\widetilde{V}, \widetilde{W}) \otimes \C^2 \to C_1(\widetilde{V}, \widetilde{W}) \otimes \C^2$
is represented by $\rho_0(A)$ for a matrix $A$ in $\Z[\pi_1 N]$,
where $\rho_0(A)$ is the matrix obtained
by naturally forgetting the submatrix structure of the matrix
whose entries are the images of those of $A$ by $\rho_0$.
The Reidemeister torsion $\tau_{\rho_0}(N, S \times 1)$ equals $\det \rho_0(A)$,
which is written as a polynomial in $\{ \tr \rho_0(A)^i \}_{i \in \Z}$,
and is also one in $\{ \tr \rho_0(\gamma) \}_{\gamma \in \pi_1 N}$.
Thus we can write
\begin{equation} \label{eq_P}
\tau_{\rho_0}(N, S \times 1) = \sum_{\gamma_1, \dots, \gamma_k \in \pi_1 N} a_{\gamma_1, \dots, \gamma_k} \tr \rho_0(\gamma_1) \dots \tr \rho_0(\gamma_k),
\end{equation}
where the sum runs over some finitely many tuples $(\gamma_1, \dots, \gamma_k)$
of elements in $\pi_1 N$ with $a_{\gamma_1, \dots, \gamma_k} \in \C$.
We set
\[ \varphi = \sum_{\gamma_1, \dots, \gamma_k \in \pi_1 N} a_{\gamma_1, \dots, \gamma_k} I_{\gamma_1} \dots I_{\gamma_k}. \]
Since the form of \eqref{eq_P} is invariant
under changes of representations $\rho_0$,
the regular function $\varphi$ satisfies \eqref{eq_F}.

Finally, we prove that $c(\mathcal{T}_C^\psi)(\chi)$ is finite.
It follows from the second step that
$c(\mathcal{T}_C^\psi)(\chi_\rho) = \tau_\rho(S, N \times 1)$
for all but finitely many irreducible representations
$\rho \colon \pi_1 M \to SL_2(\C)$ with $\chi_\rho \in C$.
Hence $c(\mathcal{T}_C^\psi)$ coincides with the regular function $\varphi$
in the third step, which is in the subring of $\C[C]$ generated
by $\{ I_\gamma \}_{\gamma \in \pi_1 N}$.
Since $\pi_1 N$ is contained in the stabilizer subgroup
of a vertex of $T_{\tilde{\chi}}$ in the construction of $S$ in Section $2$,
it follows from \cite[Theorem 2.2.1]{CS} that $I_\gamma$ does not have a pole
at $\chi$ for all $\gamma \in \pi_1 N$. 
Therefore $c(\mathcal{T}_C^\psi)(\chi) \in \C$, which completes the proof.


\end{document}